\documentclass[12pt]{article}

\usepackage{amsfonts,amsthm,amsmath,amssymb,graphicx}
\usepackage{titlesec}
\usepackage{bbm}
\usepackage{enumitem}
\usepackage[english]{babel}
\usepackage[applemac]{inputenc}
\usepackage{mathrsfs}

\DeclareMathAlphabet{\mathpzc}{OT1}{pzc}{m}{it}

\titleformat{\section}{\normalfont\scshape\center}{\thesection}{1em}{}
\titleformat{\subsection}{\normalfont\scshape}{\thesubsection}{1em}{}
\titleformat{\subsubsection}{\normalfont\scshape}{\thesubsubsection}{1em}{}

\newcommand{\N}{\mathbb N}

\newcommand{\R}{\mathbb R} 
\newcommand{\E}{\mathbb E}
\renewcommand{\P}{\mathbb P} 

\newcommand{\D}{\mathcal{D}}
\newcommand{\F}{\mathcal{F}}

\newcommand{\e}{\mathrm{e}}

\newcommand{\PP}{\mathscr{P}}

\newcommand{\8}{\infty}

\newtheorem{thm}{Theorem}[section]
\newtheorem{lem}[thm]{Lemma}

\newtheorem{prop}[thm]{Proposition}

\newtheorem{cor}[thm]{Corollary}

\numberwithin{equation}{section}

\newcounter{c}

\begin{document}
\title{Finite-size corrections to the speed of a branching-selection process
}
\author{Francis Comets \and Aser Cortines}

\maketitle

{\footnotesize 
\noindent  Universit\'e Paris Diderot -- Paris 7, 
Math\'ematiques, 
 case 7012, F--75205 Paris
Cedex 13, France
\\
\noindent e-mail:
\texttt{\{comets, cortines\}@math.univ-paris-diderot.fr}
}

\begin{abstract}
We consider a particle system studied by E. Brunet and B. Derrida \cite{Brunet2004}, which
evolves according to a branching mechanism with selection of the fittest keeping the population size fixed and equal to $N$.
The particles remain grouped and move like a travelling front driven by a random noise with a deterministic speed.
Because of its mean-field structure, the model can be further analysed as $N \to \8$.  
We focus on the case where the noise lies in the max-domain of attraction of the Weibull extreme value distribution and 
show that under mild conditions the correction to the speed has universal features depending on the tail probabilities. 
\end{abstract}

\section{Introduction  and main result}\label{Section:Introduction}

We consider a model of front propagation introduced by Brunet and Derrida \cite{Brunet2004}. A constant number $N$ of particles evolve on the real line in discrete time. Let  $X_1(0), \ldots, X_N(0)$ be the particles initial positions. With $\{ \xi_{ij} (s); 1 \leq i,j \leq N , \, s \geq 1 \}$ an i.i.d. family of r.r.v.s, the positions evolve as
\begin{equation}\label{equa: definition front propagation}
X_i(t+1) : =\max_{1 \leq j \leq N} \big\{X_j(t) + \xi_{ji} (t+1)  \big\}.
\end{equation}
For $\xi_{ij} \in L^1$, it is proved in \cite{Comets2013}  that the limits
\begin{equation} \notag
\lim_{t\to \infty} \frac{1}{t}\max_{1\leq i \leq N} X_i(t) = \lim_{t\to \infty} \frac{1}{t}\min_{1\leq i \leq N} X_i(t) = v_N(\xi)  
\end{equation}
exist a.s. with  $v_N(\xi)$ a real constant depending on the law of $\xi$. The limit $v_N(\xi)$ is called the speed of the $N$-particle system, 
and we study here its asymptotic for large $N$.
\medskip

The model can be viewed as the combination of a branching step with a fixed number $N$ of offspring per individual
each one being subject to mutation given by the addition of a fresh r.v. $\xi$, and a selection step where only the $N$ ``rightmost" among these $N^2$ offspring are kept.  
One major question of the field is to understand what really determines the motion and
derive the universality properties of such models. In particular, the speed of propagation depends 
on the parameter tuning how stringent the selection is, and one is interested in the corrections with respect to the speed of the model without selection. 
The definition of ``rightmost" used in the present paper,  see
(\ref{equation:front propagation / polymer}), is somewhat specific, and is different from the traditional choice
of $M$-branching random walks \cite{Brunet1997, Brunet1999} when all newborn individuals are simultaneously compared. 
For the latter choice we mention \cite{Berard2010}  and  \cite{Couronne2011}, and also
\cite{Mueller2011}  for the continuous case. A dual problem is the survival of the branching population killed by a moving obstacle, e.g. a line \cite{Berestycki2013}. 
In its general form, the model relates to propagations of pulled fronts, when the motion is determined by the leading edge \cite{Panja2004}. 
Archetypes of  pulled fronts are branching random walks or branching Brownian motions, described by the Kolmogorov-Petrovskii-Piscunov equation. There, and in contrast
to the present case, one looks for the second order correction {\it in  time} of the rightmost position to its leading order \cite{Bramson1983}. 
Though effective equations in the continuum are available  to describe front dynamics, the process here is intrinsically random and discrete,
adding  interest to its understanding. We note from \cite{Durrett2011} that asynchronous dynamics leads to  free boundary problems.

\medskip

 Already mentioned in \cite{Cook1990}, the model (\ref{equa: definition front propagation}) was taken up in  \cite{Brunet2004} and
 studied in the case of Gumbel distribution for $\xi$, which leads to an exact solution for fixed $N$, and  results have been extended in a 
 perturbative picture provided that $\xi$ has an
 exponential upper tail  \cite{Comets2013}.
In the present paper, we consider perturbations of the Weibull  distribution, including
bounded $\xi$'s with a polynomial density close to its maximal value.

\medskip

From a different perspective, our model can be interpreted as a first passage percolation.
%
%
By a simple induction argument, one obtains from (\ref{equa: definition front propagation}) 
the formula
\begin{equation} \label{equation:front propagation / polymer}
X_i(t)= \max \Big\{ X_{j_0}(0) + \sum_{s=1}^{t} \xi_{j_{s-1} j_s} (s); 1\! \leq \! j_s \!\leq \! N, \, \forall s = 0, \ldots, t\!-\!1 \text{ and } j_t \!= \!i \Big\},
\end{equation}
which yields 
a path representation of the interacting particle system. 
 We interpret now $-\xi_{ji}(t+1)$ as the passage time on the oriented edge from $(j,t)$ to $(i, t+1)$: 
As (\ref{equation:front propagation / polymer}) shows,   the negative of $X_i(t)$ is the passage time from the line $t=0$ to the point $(i,t)$, in a model of first passage percolation  on the vertex set $\{1,\ldots,N\} \times \N$, and $v_N$ is the so-called time constant of the model. Here the graph is oriented ($t$-coordinate increases by one 
unit at each step of the path), though on the transverse direction  jumps are allowed between all pairs of sites $j,i \ (1\leq i, j \leq N)$.  Since the graph is complete in the transverse direction, the model is of mean-field type. 
For general percolation models the value of  the time constant is not available, but in the present case the mean-field feature allows us to 
 determine the time-constant up to first order  in the limit of large graphs. In the particular case of exponential passage times, the first formula in (\ref{eq:vexp}) below is in force.

\medskip

To describe our framework, denote by $\Lambda(u)$ the logarithmic generating function of $\xi_{ij}$,
$$
\Lambda(u) := \ln \E \big[ \exp(u \xi_{ij}) \big],
$$
and let $\D_\Lambda := \{ u \in \R ; \Lambda(u) < \infty \}$ be its domain. In this paper, we will assume that the following hypothesis hold:
\begin{enumerate}[label=(H\arabic*)]
\item \label{first_hypotesis}$0 \in \D^0_\Lambda$ (the interior of $\D_\Lambda$). In particular, $\xi_{ij}$ has finite moments of all orders. 
\item  \label{second_hypotesis} For every $N \in \N$ there exists a $u_N \in \D^0_\Lambda \cap [0 , \infty )$ such that
\begin{equation}\label{equation definition u_N}
u_N \Lambda'(u_N) - \Lambda(u_N) = \ln N.
\end{equation}
\end{enumerate}
The function $u \Lambda'(u) - \Lambda(u) $ is increasing on $\D^0_\Lambda \cap [0 , \infty )$, hence $u_N$ is unique. Under these hypothesis the number 
\begin{equation}\label{equa:v_N^*}
v_N: = \Lambda'(u_N)
\end{equation} 
is well defined. If $I_\xi(v)$ is the Cramer transform of $\xi_{ij}$
$$
I_\xi (v): = \sup_{x\in \R} \big\{v x - \Lambda(x) \big\},
$$ 
then $v_N$ is determined by $I_\xi (v_N) = \ln N, v_N> \E [\xi]$, see Figure \ref{fig:Ixi}, and it holds $I_\xi'(v_N)=u_N$.

\begin{figure}
\begin{center}
\includegraphics[width=0.5\textwidth]{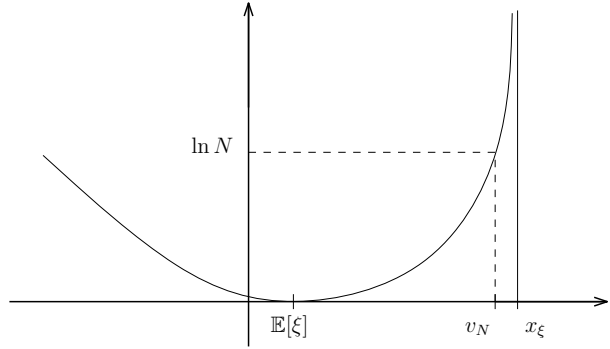}
\caption{Cramer transform $I_\xi$ and $v_N$.}
\label{fig:Ixi}
\end{center}
\end{figure}

\medskip

In Section \ref{Section upper bound to the velocity}, we show that $v_N$ is an upper bound for the velocity $v_N(\xi)$ of the $N$-particle system. To obtain a lower bound to $v_N(\xi)$, we do some additional assumptions on $\xi_{ij}$ and focus on a more restrictive class of distributions. Denoting by $F(\cdot)$ the common probability distribution function
$$
F(x) := \P\big( \xi_{ij} \leq x \big),
$$
we will further assume that $F(\cdot)$ belongs to the max-domain of attraction of the Type~III extreme value distribution with probability distribution function 
$\Psi_\alpha (\cdot)$ given by
\begin{equation}\label{equaiton_Weibull}
\Psi_\alpha (x) = \left \{
\begin{array}{l l}
    \exp \left(- |x|^{\alpha}  \right) & \text{ if } x < 0 \\
    1 & \text{ if } x \geq 0, \\
\end{array}
\right.
\end{equation}
for some $\alpha > 0$. 
This law is sometimes called reverse--Weibull  (see e.g. Chapter 1 in \cite{Resnick1987}),
or Weibull for short,  and it is the law of $- {\mathcal E}^{1/\alpha}$ with 
${\mathcal E}$ an exponential variable with mean 1.
It is well known that $F(\cdot)$ belongs to the domain of attraction of $\Psi_\alpha (\cdot)$ if and only $F(\cdot)$ has a finite right-end
\begin{equation} \notag
x_\xi := \sup \{ x \in \R; F(x) < 1 \} < \infty,
\end{equation}
and for each $x > 0$ 
\begin{equation}\label{equation_condition_typeIII_convergence} \notag
\lim_{h \to 0+} \frac{1-F\big(x_\xi - h x\big)}{1-F\big(x_\xi - h\big)} = x^\alpha,
\end{equation}
see, for example, Proposition 1.13 in Section 1.3 of \cite{Resnick1987}. In this case, let
\begin{equation}\label{equation_definition_aN}
a_N := x_\xi - \inf\{ x ; F(x) \geq 1-1/N \},
\end{equation}
then $F^N\big(x_\xi + x a_N  \big) \to \Psi_\alpha (x)$ as $N \to \infty$ and
\begin{equation}   \notag \label{equaiton_assymptotic_aN}
\lim_{N \to \infty} N \big( 1-F(x_\xi - a_N) \big) = 1
\end{equation}
The main result of this paper is the following theorem concerning the speed of the $N$ particle system.

\begin{thm}\label{thm:main_theorem} Assume that \ref{first_hypotesis}, \ref{second_hypotesis} hold, and that $\xi_{ij}$ belongs to the domain of attraction of the extreme value distribution $\Psi_\alpha$, for some $\alpha > 0$. Let
\begin{equation}\label{defintion_c_alpha}
c_\alpha := \frac{\alpha}{\e} \Big( {\Gamma(\alpha) \alpha } \Big)^{-\frac{1}{\alpha}},
\end{equation}
where $\Gamma(\cdot)$ is Euler's gamma function and $\e = 2.718\ldots$ is the  Napier's constant. Then, the speed $v_N(\xi)$ of the $N$-particle system satisfies   
\begin{equation}\notag
\displaystyle v_N(\xi) =  x_\xi - c_\alpha a_N + o \left( a_N \right) \qquad \text{as } N \to \infty, 
\end{equation}
where $a_N$ is given by (\ref{equation_definition_aN}).
\end{thm}
\medskip

Theorem~\ref{thm:main_theorem} gives a partial answer to the problem raised by E. Brunet and B. Derrida \cite{Brunet2004}, and we obtain the finite-size corrections to the speed for a large class of distributions that are bounded from above. Our result comprises, for example, the negative of the exponential distribution and the uniform distribution, for which the correction to the speed are
\begin{equation}
\label{eq:vexp}
v_N(-\mathcal{E}) \sim -\frac{1}{\e N} \quad \text{and} \quad \left( v_N(U) - 1 \right) \sim -\frac{1}{\e N}  \quad \text{as } N \to \infty,
\end{equation}
respectively. In the above formulas, ``$\sim$" means that the ratio of the sides approaches to one as $N \to \infty$.

\medskip
\textbf{Warm-up calculations:} Let us explain how to determine the order of magnitude of the correction from elementary considerations. Assume in this paragraph  that $x_\xi=0$. On the one hand, we can bound from below our $N$-particle system 
with a single particle following the leader, i.e., the random walk with jumps law given by $\max_{i \leq N} \xi_{1i}$, resulting with a lower bound for $v_N(\xi)$ of order $a_N$. On the other hand, a naive upper bound is given by the random walk with jumps $\max_{i,j \leq N} \xi_{ji}$, which leads to a different order $O(a_{N^2})$ of the correction for the maximum is over $N^2$ variables this time.
One can improve the  upper bound by using the first moment method of Section \ref{Section upper bound to the velocity},  leading to the same  order $O(a_N)$ as the lower bound. However, the 
multiplicative factors do not match, and some deeper understanding and improvement of the lower bound is needed. 
This is what we implement in Section \ref{Section lower bound}, using a comparison with a branching random walk with selection.

\medskip
\textbf{Organization of the paper:} In Section~\ref{section_general_results}, we present some point processes and branching random walks related to our model
and we sumarize the necessary results for our purpose. 
We prove the upper bound for the speed in Section~\ref{Section upper bound to the velocity} by a first moment estimate, and 
the  lower bound in Section~\ref{Section lower bound} by coupling, the two bounds resulting in 
Theorem~\ref{thm:main_theorem}.

\section{Point processes and branching random walks}\label{section_general_results}

In this section, we introduce different processes  entering the analysis of the particle system (\ref{equa: definition front propagation}).
 
\subsection{Point measures on $\R$}

It is convenient to represent populations of particles by point measures on $\R$. Given a vector $x \in \R^n$ with coordinates $x_1, \ldots, x_n$, one can associate the point measure 
$$
\mathpzc{x} := \sum_{i=1}^{n} \delta_{\{ x_i\}}.
$$ 
We use the notation $\mathcal{M}_b$ to represent the set of all simple point measures on $\R$, which are  locally finite and have a maximum.
Throughout this paper, a point process is any random variable $\mathcal{L}$ taking values on $\mathcal{M}_b$. 

Conversely, an element $\nu \in \mathcal{M}_b$ can be described
as a sequence $\nu = (\nu_i)_{i=1, 2 , \ldots }$ (possibly finite) such that 
$$
\nu_1 \geq \nu_2 \geq \ldots 
$$
We denote by $\max (\nu) = \nu_1$ the maximum of the support of $\nu$, and by $|\nu| = \nu(\R)
\leq \8$ the number of points in $\nu$. 
If two point measures $\nu$ and $\mu$ have the same number of points $|\nu| = |\mu| = K$, we can define the distance
\begin{equation}\label{measure distance}
\| \nu- \mu \| = \sup_{1 \leq i \leq K} \{ |\nu_i - \mu_i| \}.
\end{equation} 
We use the notation ``$\prec$" to denote the usual stochastic ordering
\begin{equation} \notag
\nu \prec \mu \quad \text{if and only if} \quad \nu [ x, +\infty) \leq \mu [ x, +\infty); \qquad \forall x \in \R,
\end{equation}
and we will say that ``$\nu$ bounds $\mu$ from below", in this case $|\nu| \leq |\mu|$. If we represent $\nu$ and $\mu$ as an ordered sequence of points, then $ \nu \prec \mu$ implies that 
$$
\nu_i \leq \mu_i \qquad \text{for every } i \leq |\nu|.
$$  
With a slight abuse of notation we will say that the vector $x\in \R^n$ bounds $y \in \R^m$ from below and denote ``$x \prec y$" if the point measures $\mathpzc{x}, \mathpzc{y}$ associated to $x$ and $y$ respectively satisfy $\mathpzc{x} \prec \mathpzc{y}$.

\subsection{Poisson point processes on $]-\infty, 0]$}

In this section, we present some elementary facts concerning Poisson Point Process 
$$
\PP=\big\{ \PP_1  > \PP_2 > \ldots \big\} \subset \R_-\,,
$$
with intensity measure 
$ |z|^{\beta} C \text{d} z$ on $\R_-$; we use the abbreviation PPP and assume that $C>0, \beta >-1$. 
For $K \geq 1$, the point process 
\begin{equation}
\label{def:PPK}
\PP^{(K)}:=\big( \PP_i \big)_{i\leq K}
\end{equation}
consisting in the $K$ largest points of $\PP$ will play an important role in the next sections.

For ${\mathcal L}$ a random point measure on $\R_-$, we denote by $\psi \big(u \mid {\mathcal L} \big)$ its logarithmic moments generating function 
\begin{equation} \notag
\psi \big(u \mid  {\mathcal L} \big) := \ln \E \left[ \int \e^{u y}  {\mathcal L} \big(\text{d}y \big)  \right],
\end{equation}
We can easily compute the logarithmic generating function of the PPP and of its $K$-truncation.

\begin{lem}\label{Poisson_lemma} For $\beta > -1, C>0$, let $\PP$ be the Poisson point process on $(-\infty, 0]$ with intensity measure $\mu (\text{d} z)= |z|^{\beta} C \text{d} z$, and $\PP^{(K)}$ its largest $K$  points.
 For $u > 0$ we have
$$
\E \left[ \int_{-\infty}^{0} \e^{u z} \PP (\text{d} z) \right] = \frac{\Gamma (1+\beta) C }{u^{1+\beta}}, 
\quad 
 \quad
\E \left[ \int_{-\infty}^{0} z \e^{u z} \PP (\text{d} z) \right] = \frac{-\Gamma (2+\beta) C }{u^{2+\beta}},  
$$ 
and
$$
\lim_{K \to \infty} \E \left[ \sum_{i=1}^{K} \e^{u \PP_i } \right] =  \frac{\Gamma (1+\beta) C }{u^{1+\beta} }, 
\quad 
 \quad
\lim_{K \to \infty} \E \left[ \sum_{i=1}^{K} \PP_i \e^{u \PP_i } \right] = \frac{-\Gamma (2+\beta) C }{u^{2+\beta}}. 
$$ 
\end{lem}
\begin{proof}
The first claim is obtained by the Campbell formula (see Chapter 9 in \cite{Daley2003}) and the second claim is obtained by monotone convergence.
\end{proof}

\begin{cor} \label{cor:PPK}
Under the assumptions of Lemma \ref{Poisson_lemma}, the sequence  $\psi \big(u \mid \PP^{(K)} \big)$ converges uniformly on the compacts of $\R_+$ to $\psi \big(u \mid \PP \big)$ as $K\to \infty$. In particular, if $u_K > 0 $ is such that
$$
\psi \big(u_K \mid \PP^{(K)} \big) = \psi' \big(u_K \mid \PP^{(K)} \big)u_K
$$
then 
\begin{equation} \label{equation_corollary_limit}
\lim_{K \to \infty} \psi ' \big(u_K \mid \PP^{(K)} \big)= - \frac{1+\beta}{\e}\left(\frac{1}{C \Gamma(1+\beta) } \right)^{1/1+\beta}.
\end{equation}
\end{cor}
\begin{proof}
The compact convergence is a direct consequence of the pointwise convergence together with the monotonicity of $\psi \big(u \mid \PP^{(K)} \big)$ and the continuity of $\psi \big(u \mid \PP \big)$
in $u$ (Dini's theorem). Let 
$$
u_\infty = \e \left( \Gamma(1+\beta) C  \right)^{1/1+\beta} 
$$
then from the first part of Lemma~\ref{Poisson_lemma} we have that 
$$
\psi \big(u_\infty \mid \PP\big) = \psi' \big(u_\infty \mid \PP \big) u_\infty \quad \text{and} \quad \psi '\big(u_\infty \mid \PP \big) = - \frac{1+\beta}{\e}\Big({C \Gamma(1+\beta) } \Big)^{-1/1+\beta}. 
$$
By point 2 of Lemma \ref{Poisson_lemma}, $u_K \to u_\8$ and the second claim follows from the uniform convergence.
\end{proof}

\subsection{Branching Random Walks
}

Branching Random Walks (BRW for short) have been extensively studied in the past years, see the seminal paper \cite{Aidekon2013} for a general literature and important results on the subject. In this paper, we focus on BRW defined as follows. Let $\mathcal{L}$ be a point process on $\R$, which defines the positions of particles and the reproductive law of the underlying Galton-Watson tree. The process starts with one particle located at 0. At each time step $t \to t+1 $, the particles of generation $t$ die and give birth to independent copies of $\mathcal{L}$, translated by their position. We use the notation BRW$(\mathcal{L})$ to denote a BRW defined as above. 

Let $\mathbb{T}$ be the Galton-Watson tree obtained by the genealogical tree of the process, thus, its offspring distribution is $|\mathcal{L}|$. To each point (or individual) of the BRW$(\mathcal{L})$ one can associate a unique vertex $w \in \mathbb{T}$. Let $e \in \mathbb{T}$ be the root of the Galton-Watson tree, then for a vertex $w \in \mathbb{T}$, let $ [\![ e, w ]\!]$ denote the shortest path connecting $e$ with $w$, and $|w|$ the length of this path. 
We will sometimes write its points $ [\![ e, w ]\!]=(e,w_1,\ldots,w_k)$ with $i=|w_i|$ and $w_k = w$.
It is a standard property of Galton-Watson trees that the process starting from a vertex $w \in \mathbb{T}$ is also a Galton-Watson tree with the same distribution. 
For two vertices $w$ and $ w'$  in $\mathbb{T}$
we denote by $w  w'$ the vertex in $\mathbb{T}$ in generation $|w| +|w'| $ obtained by concatenation.

We also denote by $\eta(w)$ the positions of the individual $w \in \mathbb{T}$, and by $\mathpzc{y}(t)$ the point measure associated to the BRW$(\mathcal{L})$
$$
\mathpzc{y}(t) := \sum_{w \in \mathbb{T}; |w|= t} \delta_{\{ \eta(w) \}}.
$$
Finally, an infinite ray $[\![e, w_\infty ]\!]:=\{e, w_1, w_2, \ldots \} \subset \mathbb{T}$ is an infinite collection of vertices (or infinite path), such that $w_i$ is the parent of $w_{i+1}$. It represents a family branch in the BRW that has not extinguished, and is parametrized by an element $w_\infty \in \partial  \mathbb{T}$ of the 
topological boundary  $ \partial \mathbb{T}$ of the tree.

Under mild conditions on $\mathcal{L}$, the asymptotic behaviour of $\max \left( \mathpzc{y}(t) \right)$ is known 
\cite{Biggins77, Athreya2004}, that we recall now. Assume that 
for some $a> 0$,
\begin{equation}    \label{equation reproductive assumption}
\E\big[ |\mathcal{L}|^{1+a}\big] < \infty,
\end{equation} 
a condition which can be weakened \cite{Aidekon2013}, but in this paper we will always have 
$|\mathcal{L}| = K$ a constant, which  trivially implies (\ref{equation reproductive assumption}). We also assume that the logarithmic generating function for the branching random walk
\begin{equation} \label{equation log generating function assumption}
\psi \big(u \mid \mathcal{L} \big) := \ln \E \left[ \int \e^{u y} \mathcal{L} \big(\text{d}y \big)  \right]
\end{equation} 
is finite in a neighbourhood of $u=0$ and that there exists a $u^* = u(\mathcal{L} ) >0$ for which
\begin{equation} \label{equation log generating function relation}
\psi \big(u^* \mid \mathcal{L} \big) = u^* \psi' \big(u^* \mid \mathcal{L} \big).
\end{equation}
If (\ref{equation reproductive assumption} -- \ref{equation log generating function relation}) hold, there exists a constant $\gamma (\mathcal{L})$ depending only on $\mathcal{L}$ such that
\begin{equation}\label{equation convergence classical BRW}
\lim_{t \to \infty} \frac{1}{t} \max \left( \mathpzc{y}(t) \right) = \gamma (\mathcal{L}) \qquad a.s.
\end{equation}
Moreover, the asymptotic speed $\gamma(\mathcal{L})$ is explicit and given by $\gamma(\mathcal{L}) = \psi ' \left( u^* \mid \mathcal{L} \right)$, see \cite{Aidekon2013, Athreya2004, Gantert2011} for a rigorous proof and more results on the subject. 

The theorem below, proved by Gantert, Hu and Shi \cite{Gantert2011}, gives the precise decay for the probability that there exists an infinite ray in the BRW that always stays close to $\gamma(\mathcal{L}) $.

\begin{thm}[\cite{Gantert2011}]\label{theorem gantert hu shi} Let $\mathcal{L}$ be a point process satisfying  (\ref{equation reproductive assumption} -- \ref{equation log generating function relation}) and $\big( \eta(w); w \in \mathbb{T}\big)$ be the BRW$(\mathcal{L})$. Given $\delta > 0$, denote by $\rho(\infty, \delta)$ the probability that there exists an infinite ray in the branching random walk that always lies above the line of slope $ \gamma(\mathcal{L}) - \delta$.
$$
\rho(\infty, \delta) := \P \Big( \exists  w_\infty \in \partial {\mathbb T}: \eta(w_t) \geq (\gamma(\mathcal{L}) - \delta) t, \, \forall w_t \in [\![e, w_\infty ]\!] \Big),
$$ 
where $w_t \in [\![e, w_\infty ]\!] $ is the vertex in generation $t$. Then, as $\delta \searrow 0$ 
$$
\rho(\infty, \delta) \sim \exp \left( - \left[ \frac{\chi(\mathcal{L}) +o(1)}{\delta} \right]^{1/2}\right),
$$
where $\chi(\mathcal{L}) = u^* \psi'' \big(u^* \mid \mathcal{L} \big)$ for $u^*$ given by (\ref{equation log generating function relation}).
\end{thm}

\subsection{Models with selection and the $M$-BRW}

Recently, some models of evolving particle systems under the effect of selection have been studied \cite{Berard2010, Cortines2013, Cortines2014, Couronne2011, Maillard2011, Mallein2015}. The selection creates correlation between individuals of same generation and additional dependence in the whole process.

B\'erard and Gou\'er\'e \cite{Berard2010} focused on the binary Branching Random Walk with selection of the $M$ rightmost individuals (the $M$-BRW). It consists in a BRW, subject to the effects of selection, defined by the point process 
$$
\mathcal{L} = \delta_{\mathpzc{p}_1}+  \delta_{\mathpzc{p}_2},
$$ 
where $\mathpzc{p}_i$ are i.i.d. As soon as the population size exceeds $M$, we only keep the $M$ rightmost individuals and eliminate the others. If at time zero the number of particles is already $M$, the population size is kept constant.

Denote by $\mathpzc{y}_M (t)$ the point process generated by this $M$-BRW. B\'erard and Gou\'er\'e \cite{Berard2010} show that the support of $\mathpzc{y}_M (t)$ has a diameter of order $\ln M$. They also prove that under some assumptions on the exponential moment of $\mathpzc{p}_i$ there exists a constant $\gamma_M(\mathcal{L})$ such that 
$$
\lim_{t \to \infty }t^{-1} \min \left( \mathpzc{y}_M (t) \right)=\lim_{t \to \infty }t^{-1} \max \left( \mathpzc{y}_M (t) \right) = \gamma_M(\mathcal{L}) \qquad a.s.
$$ 
The existence of $\gamma_M(\mathcal{L})$ is obtained by the Kingman's sub-additive ergodic theorem, and by monotonicity arguments one can prove that $\gamma_M(\mathcal{L})$ converges as $M\to \infty$. The striking result is that B\'erard and Gou\'er\'e computed the asymptotic limit and the rate of convergence:
$$
\gamma_M(\mathcal{L}) = \gamma(\mathcal{L}) + \chi (\mathcal{L}) (\ln M)^{-2} + o \big( (\ln M)^{-2} \big) \qquad \text{as } M \to \infty,
$$  
where $\gamma(\mathcal{L})$ is the asymptotic speed for the BRW$(\mathcal{L})$, and $ \chi (\mathcal{L})$ is from Theorem~\ref{theorem gantert hu shi}.

One can easily define more general $M$-BRW, let $\mathcal{L}$ be a point process, for which 
\begin{equation}\label{non extinction condition}
|\mathcal{L}| \geq 1 \qquad a.s.
\end{equation} 
so the process does not extinguish. Under similar assumptions on the exponential moments of $\mathcal{L}$, one can prove that the cloud of $M$ points does not spread and that it travels at a certain speed $\gamma_M (\mathcal{L})$, see \cite{Mallein2015}.

In Section~\ref{Section lower bound}, we show that under the hypothesis of the Theorem~\ref{thm:main_theorem}, the $N$-particle system (\ref{equa: definition front propagation}) can be bounded from below by a family of $M$-BRW indexed by $N$. We will then adapt the arguments in \cite{Berard2010} to obtain a uniform lower bound for the speeds of the BRWs.

\subsection{Elementary properties of Brunet-Derrida's $N$-particle system} \label{Section:elementary_properties}

In this section, we present some elementary properties of the $N$-particle system. We also introduce some notations that will be useful in the forthcoming sections. Most of these properties have been rigorously proved in \cite{Comets2013}, therefore we will simply outline the main ideas. 

It will be convenient to consider the process $ X^* (t)$ obtained by ordering the components of $X(t)$ at each time $t$. Denote by 
\begin{equation} \notag
X^{(1)}(t) \geq X^{(2)}(t) \geq \ldots \geq X^{(N)}(t) 
\end{equation}
the components of $ X^* (t)$. Let, also $\sigma =\sigma(t)$ be the random permutation of $\{ 1, \ldots, N \}$ such that
\begin{equation} \notag
X^{(i)} (t) = X_{\sigma_i (t)} (t).
\end{equation}
Such a ranking permutation is unique up to ties, which we break in the order of the original labels.
Consider the random variable
$$
T := \inf \Big\{ t \geq 1 ; \, \xi_{\sigma_1(t-1), i} (t) = \max_{1 \leq j \leq N} \{\xi_{ji} (t)\}; \, \forall i = 1 \ldots  N  \Big\}.
$$
Then, $T$ is a stopping time for the filtration 
\begin{equation}
\label{eq:filtration}
\F_t := \sigma \big\{  X_i(0) , \, \xi_{i j} (s); \, 1 \leq i, j \leq N \text{ and } s \leq t   \big\}.
\end{equation}
It has a geometric distribution with parameter not smaller than $(1/N)^N$. Moreover, in generation $T$ the position of each particle $X_i(T)$ is determined by the position of the leader $X_{\sigma_1}$ in generation $T-1$. We define the process seen from the leading edge
$$
X_i^{\circ}(t) := X_i(t) - X_{\sigma_1(t)}(t).
$$
It is Markov process on $\R^N$, which is irreducible, aperiodic and Harris recurrent (due to the renewal structure), thus there exists a unique stationary measure $\pi$, and for any starting point $X(0)= x$ the law of $X(t)$ converges in total variation distance,
$$
\text{dist}_{\text{T.V.}} \big( \mathscr{L} (X (t) \mid X(0)= x) , \pi \big) \to 0, \quad \text{as } t \to \infty.
$$
It proves, in particular, that the cloud of $N$ points remains grouped as $t \to \infty$. Moreover, by the renewal and ergodic theorems, the limit
\begin{equation}\notag
v_N(\xi) = \lim_{t \to \infty} \frac{1}{t} \max_{1 \leq i \leq N} \{ X_i(t) \} = \lim_{t \to \infty} \frac{1}{t} \min_{1 \leq i \leq N} \{ X_i(t) \} 
\end{equation}
exists a.s., see Section~2 in \cite{Comets2013} for a rigorous proof and  more details.

\section{Upper bound for the velocity} \label{Section upper bound to the velocity}

In this section, we show that if $\xi_{ij}$ satisfies the hypothesis \ref{first_hypotesis}, \ref{second_hypotesis}, then
$$
v_N(\xi) \leq v_N =  \Lambda'(u_N),
$$ 
where $u_N >0$ is the unique positive solution of $u \Lambda' (u) - \Lambda (u) = \ln N$. The idea is to use the so-called first moment method to bound the probability
$$
\P\left(\max_{1\leq i \leq N} \big\{ X_i(t) \big\} > t \Lambda'(u_N) \right).
$$

A first and simple observation is that the initial position of the particles does not change the speed of the $N$-particle system. Hence, we may assume without loss of generality that all $N$ particles start at zero. Using the representation (\ref{equation:front propagation / polymer}) one gets
$$
\max_{1\leq i \leq N} \big\{ X_i(t) \big\} = \max \left\{ \sum_{s=1}^{t} \xi_{j_{s-1} j_s} (s); 1\! \leq \! j_s \!\leq \! N, \, \forall s = 0, \ldots, t \right\}.
$$
By the union bound and Chernoff bound we obtain, for $v \!>\! v_N \!=\! \Lambda' (u_N)$ (which is larger than~$\E [\xi]$ for $N$ sufficiently large), 
\begin{align} \label{equa proof for upper bound for the maximum}
\P\left(\max_{1\leq i \leq N} \big\{ X_i(t) \big\} \geq t v \right) & =  \P\left( \exists {j_0, j_1, \ldots, j_t}: \sum_{s=1}^{t} \xi_{j_{s-1} j_s} (s) \geq t v    \right) \notag \\
& \leq N^{t+1} \P\left( \sum_{s=1}^{t} \xi_{j_{s-1} j_s} (s) \geq t v  \right) \notag \\
& \leq N^{t+1} \exp \big(-t I_\xi(v) \big) ,
\end{align}
for all $N \in \N$. Since \ref{first_hypotesis} and \ref{second_hypotesis} hold, $I_\xi(v)$ exists and $I_\xi(v) > \ln N$. As a consequence, (\ref{equa proof for upper bound for the maximum}) has a geometrical decay as $t \to \infty$, which implies, by Borel-Cantelli lemma,
$$
\P\left(\max_{1\leq i \leq N} X_i(t) \geq t v \text{ for infinitely many } t\in \N \right) = 0,
$$
hence, $\limsup_{t \to \infty} t^{-1} \max\{ X_i(t) \} \leq v$ a.s. for every $v>v_N$, finally yielding
\begin{equation}\label{equa upper bound velocity}
\limsup_{t \to \infty} \frac{1}{t}\max_{1\leq i \leq N} X_i(t) \leq  v_N \qquad a.s.
\end{equation}
We formalize this result in a proposition.

\begin{prop} Assume that \ref{first_hypotesis}, \ref{second_hypotesis} hold. Let $u_N >0$ such that $u_N \Lambda' (u_N) - \Lambda (u_N) = \ln N$ and $v_N = \Lambda'(u_N)$. Then, for every $N \in \N$,
$$
v_N(\xi) \leq v_N.
$$
\end{prop}


Hence the next step is to study the asymptotic of $v_N$, that we start with the case $x_\xi=0$. 

\begin{prop}\label{prop:asymptotic_vn} Assume that the hypothesis of Theorem~\ref{thm:main_theorem} hold with $x_\xi = 0$. Let $u_N > 0$ be the unique solution of $ u_N \Lambda' (u_N) - \Lambda (u_N) = \ln N$, $c_\alpha$ be given by (\ref{defintion_c_alpha})
and $a_N$ by (\ref{equation_definition_aN}).
Then, as $N \to \infty$
$$
\Lambda ' (u_N) = -c_\alpha a_N + o \left( a_N \right),
$$
which implies that $\limsup_{N \to \infty } a_N^{-1} v_N (\xi)  \leq - c_\alpha $.
\end{prop}  
\begin{proof} By definition of $\Lambda(\cdot)$, we have that
$$
\frac{\E\big[u_N \xi_{ij} \e^{u_N \xi_{ij}} \big]}{\E\big[ \e^{u_N \xi_{ij}} \big]} - \ln \left( \E\big[ \e^{u_N \xi_{ij}} \big] \right) = \ln N.
$$
Note that $u_N \to \infty$ as $N \to \infty$, indeed it is a direct consequence of the monotonicity and continuity of $ u \Lambda' (u) - \Lambda (u)$. Hence, the asymptotic behaviour of the Laplace transform of $\xi_{ij}$ in $u_N$ is determined by its behaviour in a neighbourhood of zero. Since $\xi_{ij}$ is in the domain of attraction of $\Psi_\alpha$, the function  $1-F(-x) \,: \R_+ \to \R_+$ is $\alpha$-regularly varying at zero. By Karamata's representation (see Chapter~0 in \cite{Resnick1987}) 
$$
1 - F(-x) = \P(\xi_{ij} > -x ) = x^\alpha c(x^{-1}) \exp \left( \int_1^{x^{-1}} \frac{\epsilon(t)}{t} \text{d}t \right), \qquad x>0,
$$
where $c(\cdot)$ and $\epsilon(\cdot)$ are positive functions such that $c(t) \to c > 0 $ and $\epsilon(t) \to 0$ as $t \to \infty$. As a consequence, given $\varepsilon > 0$, one can find a $u_\varepsilon > 0$ such that for $0< u \leq u_\varepsilon$
$$
1-F(-u) \geq (c -\varepsilon ) u^{\alpha}.
$$ 
Now, we compute the Laplace transform of $\xi_{ij}$ in $u_N$
\begin{align*}
\E\big[ \e^{u_N \xi_{ij}} \big] & = \P\big(\xi_{ij} u_N  \geq -1 \big) \int_{0}^{\infty} \e^{-z} \frac{\P\big(\xi_{ij} u_N  \geq -z \big)}{ \P\big(\xi_{ij} u_N  \geq -1 \big) } \text{d}z \\
& = \big(1-F\left(-u_N^{-1} \right) \big) \left( \int_{0}^{\sqrt{u_N}} \cdots \, \text{d}z + \int_{\sqrt{u_N}}^{\infty} \cdots \, \text{d}z \right).
\end{align*}
We analyse each integral separately. For $N$ sufficiently large $u_N^{-1} \leq u_\varepsilon$, hence
$$
\int_{\sqrt{u_N}}^{\infty} \e^{-z} \frac{\P\big(\xi_{ij} u_N  \geq -z \big)}{ \P\big(\xi_{ij} u_N  \geq -1 \big) } \text{d}z  \leq \frac{u_N^{\alpha}}{(c -\varepsilon) } \int_{\sqrt{u_N}}^{\infty} \e^{-z}  \text{d}z,  
$$
which converges to zero as $N \to \infty$. Take $L > 0$, and assume that $\sqrt{u_N} > L$, then
$$
\int_{0}^{ \sqrt{u_N} } \e^{-z} \frac{\P\big(\xi_{ij} u_N  \geq -z \big)}{ \P\big(\xi_{ij} u_N  \geq -1 \big) } \text{d}z = \int_{0}^{L} \cdots \, \text{d}z + \int_{L}^{\sqrt{u_N}} \cdots \, \text{d}z .
$$
Using dominated and monotone convergence we obtain that
$$
\lim_{L \to \infty} \lim_{N \to \infty} \left(  \int_{0}^{L} \e^{-z} \frac{\P\big(\xi_{ij} u_N  \geq -z \big)}{ \P\big(\xi_{ij} u_N  \geq -1 \big) } \text{d}z  \right) = \Gamma(\alpha +1).
$$
Finally, we prove that the integral from $L$ to $\sqrt{u_N}$ vanishes as $N\to \infty$ and $L \to \infty$ (in this order). For $L > 1$ and $L \leq z \leq \sqrt{u_N}$, Karamata's representation yields
$$
\frac{\P\big(\xi_{ij} u_N  \geq -z \big)}{ \P\big(\xi_{ij} u_N  \geq -1 \big) } = z^\alpha \frac{c(z^{-1} u_N)}{ c(u_N)}   \exp \left( \int_{u_N/z}^{u_N} \frac{\epsilon(t)}{t} \text{d}t \right).
$$
Taking $N$ sufficiently large so $\epsilon(t) < \varepsilon$ and $|c-c(t)| \leq \varepsilon$ for every $t \geq \sqrt{u_N}$
$$
\frac{\P\big(\xi_{ij} u_N  \geq -z \big)}{ \P\big(\xi_{ij} u_N  \geq -1 \big) } \leq z^\alpha \frac{(c + \varepsilon) }{ (c- \varepsilon)} z^{\varepsilon},
$$
which yields the upper bound
$$
\int_{L}^{ \sqrt{u_N} } \e^{-z} \frac{\P\big(\xi_{ij} u_N  \geq -z \big)}{ \P\big(\xi_{ij} u_N  \geq -1 \big) } \text{d}z   \leq \frac{(c + \varepsilon) }{ (c- \varepsilon)}  \int_{L}^{ \infty } \e^{-z} z^{ \alpha + \varepsilon} \text{d}z.
$$
The right-hand side of this inequality decays to zero as $L \to \infty$, and hence
$$
\E\big[ \e^{u_N \xi_{ij}} \big] \sim \big(1 - F\left(-u_N^{-1} \right) \big) \Gamma(1+ \alpha) \quad \text{ as } N \to \infty.
$$
By a similar argument one obtains that
$$
\E\big[ u_N \xi_{ij}  \e^{u_N \xi_{ij}} \big] \sim \big(1 - F\left(-u_N^{-1} \right) \big) \big( \Gamma(1+ \alpha) -\Gamma(\alpha+2) \big) \quad \text{ as } N \to \infty.
$$
The formula $u_N \Lambda'(u_N) - \Lambda (u_N) = \ln N$ yields 
$$
\left( 1 - F\left(-u_N^{-1}\right) \right) N \sim \frac{1}{\e^{\alpha} \alpha \Gamma(\alpha)} \quad \text{as } N \to \infty.
$$
We now use (\ref{equaiton_assymptotic_aN}) and Karamata's representation to conclude that 
$$
\lim_{N \to \infty} \frac{u^{-1}_N}{a_N} = \frac{1}{\e} \left(\frac{1}{\alpha \Gamma(\alpha)} \right)^{1/\alpha},
$$
and hence as $N \to \infty$
$$
\Lambda' (u_N)  \sim - \frac{\alpha}{u_N}  \sim - \frac{\alpha}{\e} \left(\frac{1}{\alpha \Gamma(\alpha)} \right)^{1/\alpha} a_N , 
$$
which proves the statement. The second claim is a direct consequence of (\ref{equa upper bound velocity}).
\end{proof}

If $x_\xi \neq 0$, we can simply translate the $\xi_{ij}$ by $x_\xi$, so the hypothesis of Proposition~\ref{prop:asymptotic_vn} hold. In the next corollary, we prove the upper bound in Theorem~\ref{thm:main_theorem}.  

\begin{cor}\label{cor:asymptotic_velocity} Assume that the hypothesis of Theorem~\ref{thm:main_theorem} hold. Then, as $N \to \infty$ 
$$
\limsup_{N \to \infty} (v_N(\xi)- x_\xi ) a_N^{-1} \leq - c_\alpha. 
$$
\end{cor}  
\begin{proof} In the case $x_\xi = 0$, it is a straightforward consequence of Proposition~\ref{prop:asymptotic_vn} and (\ref{equa upper bound velocity}). If $x_\xi \neq 0$, it suffices to translate the variables $\xi_{ij}$ by $x_\xi$.
\end{proof}

\medskip

\section{Lower bound}\label{Section lower bound}

In this Section, we show that for every $\varepsilon >0$ there exists a  $N_0$ such that $ \forall N \geq N_0$
\begin{equation} \label{equation lowerbound epsilon}
\frac{(v_N(\xi) -x_\xi)}{a_N}\geq - c_\alpha - \varepsilon, 
\end{equation}
which proves the lower bound in Theorem~\ref{thm:main_theorem}. Throughout this section, we fix
an arbitrary  $\varepsilon >0$, and we assume that $x_\xi = 0$ without loss of generality.
To prove (\ref{equation lowerbound epsilon}), we construct a process $x(t) \in \R^{M}$ that bounds $X(t)$ from below, hence
$$
\max \big(x(t) \big) \leq \max \big(X(t) \big).
$$
Then, in Subsection \ref{subsection coupling}, we check that the process $x(t)$ is a M-BRW, and we show that for $M$ large enough and the appropriate offspring distribution (see Subsection \ref{subsection velocity BRW}),
$$
\liminf_{t\to \infty} \frac{1}{t} \max \big(x(t) \big) \geq -(c_\alpha + \varepsilon)a_N \qquad a.s.
$$ 
which implies (\ref{equation lowerbound epsilon}) for $x_\xi = 0$. The general case is obtained by a simple affine transformation.
\medskip

\subsection{Coupling with a Branching Random Walk}\label{subsection coupling}
 
We construct $x(t)$ inductively as follows: let $M, K \in \N$, the appropriate values for $K$ and $M$ will be chosen later on, and assume that $N  \geq KM$ (in fact, we will take $KM$ negligible compared to $N$). For $t=0$, we define 
$$
x_i (0) = X_{\sigma_i} (0),
$$ 
with $\sigma_i = \sigma_i(0)$. Assuming that the process $x(\cdot)$ has been constructed up to time $t \in \N$, the vector $x(t+1) \in \R^M$ is obtained according to the inductive rule below. 

\begin{enumerate}
\item \emph{Branching step:} Every particle $x_i(t)$ is replaced by $K$ new particles (reproductive law), whose positions are defined by a point process $\mathcal{L}^{(K)} \big( x_i(t) \big)$ translated by $x_i(t)$.

The point processes $\left( \mathcal{L}^{(K)} \big( x_i(t) \big); \ 1 \! \leq \! i \! \leq \! M \right)$ are also constructed according to an inductive rule, that we describe:  
\begin{itemize}
\item {\it For $i=1$}, let $\mathcal{T}_1 := \{1, \ldots, N - K M \}$ and denote by 
$$ 
\xi_{\sigma_1(t)}^{(1:\mathcal{T}_1)} (t+1) \geq \xi_{\sigma_1 (t)}^{(2:\mathcal{T}_1)} (t+1) \geq \ldots \geq \xi_{\sigma_1 (t) }^{(K:\mathcal{T}_1)} (t+1), 
$$ 
the $K$ largest values among $\big\{ \xi_{\sigma_1 (t), j} (t+1) ; \,  j \in \mathcal{T}_1  \big\}$. Let, also,
$$
\mathcal{J}_1 = \mathcal{J}_1 (t+1) : = \{ j_1^{(1)}, \ldots, j_K^{(1)} \}
$$ 
be the set of their indices, that is,
$$
\xi_{\sigma_1(t) }^{(l:\mathcal{T}_1)} (t+1) = \xi_{\sigma_1 (t) , j^{(1)}_l} (t+1),
$$
we will keep track of these labels. Note that the indices $j^{(1)}_l = j^{(1)}_l (t\!+\!1); \ 1\leq \! l \! \leq K$ are random. Then, $\mathcal{L}^{(K)}\big( x_1(t) \big)$ is the point process 
$$
\mathcal{L}^{(K)}\big( x_1(t) \big) := \sum_{j \in \mathcal{J}_1 (t+1)} \delta_{ \left\{ \xi_{\sigma_{1}  j } (t+1) \right\} },
$$
obtained by the $K$ largest values in $\mathcal{T}_1$ and the descendants of $x_1(t)$ are at the positions:
$$
x_1(t) + \xi_{\sigma_1 (t)}^{(l: \mathcal{T}_1)} (t+1) \qquad \text{ for } 1 \leq l \leq K.
$$

\item {\it Assume that we have constructed $\left( \mathcal{L}^{(K)}\big( x_j(t) \big); 1\! \leq \! j \! \leq \! i-1 \right)$ and that we have kept track of the sets $\mathcal{J}_{1}, \ldots,\mathcal{J}_{i-1}$, appearing in the respective constructions. The sets $\mathcal{J}_{j} = \mathcal{J}_{j} (t+1) \subset \{1,\ldots, N\} $ are random and disjoint.} Then, given $\mathcal{J}_1 \cup \ldots \cup \mathcal{J}_{i-1}$, we choose
$$
\mathcal{T}_{i} = \mathcal{T}_{i}(t+1) \subset \{1, \ldots, N\} \setminus \big(\mathcal{J}_1 \cup \ldots \cup \mathcal{J}_{i-1}  \big)
$$ 
according to a deterministic rule. For example, one can choose the $N-MK$ first elements (in the usual order of $\N$) in $\{1, \ldots, N\} \setminus \big(\mathcal{J}_1 \cup \ldots \cup \mathcal{J}_{i-1}  \big)$. By construction, $\mathcal{T}_i$ is a random set of $\{1,\ldots, N \}$ satisfying the property
$$
\mathcal{T}_i \cap \mathcal{J}_1 = \emptyset = \mathcal{T}_i \cap \mathcal{J}_2 = \ldots = \mathcal{T}_i \cap  \mathcal{J}_{i-1}.
$$  
Let 
$$
\xi_{\sigma_i(t)}^{(1: \mathcal{T}_i)}  (t+1)\geq \xi_{\sigma_i (t)}^{(2: \mathcal{T}_i )} (t+1) \geq \ldots \geq \xi_{\sigma_i (t) }^{(K: \mathcal{T}_i)} (t+1)
$$ 
be the $K$ largest values among $\big\{ \xi_{\sigma_i (t) ,j} (t+1) ; \, j \in \mathcal{T}_i \big\}$, and $\mathcal{J}_i(t+1) = \{ j_1^{(i)}, \ldots, j_K^{(i)} \}$ be the set of their indices, that is,
$$
\xi_{\sigma_i(t)}^{(l: \mathcal{T}_i)}  (t+1) = \xi_{\sigma_i(t) j_l^{(i)} } (t+1).
$$ 
Then, $\mathcal{L}^{(K)}\big( x_i(t) \big)$ is the point process formed by the these $K$ points. 
\end{itemize}   

We end up the branching step with $K M$ new particles.

\item \emph{Selection:} We select the $M$ rightmost particles among the $K M$ obtained in the branching step.

\item \emph{Ordering:} We reorder the $M$ selected particles to obtain the vector $x(t+1)$.
\end{enumerate}

In the next two lemmas, we show that $x(t)\prec X(t)$ and that $\mathcal{L}^{(K)} (\cdot)$ are i.i.d. which implies that the point process
$$
\mathpzc{x}(t) := \sum_{i=1}^{M} \delta_{\{ x_i(t)\}}
$$
has the distribution of the point process obtained from a M-BRW $\big(\mathcal{L}^{(K)} \big)$.

First, we prove that $x(t)$ bounds $X(t)$ from below. We bring to the reader's attention that the next lemma is a direct corollary of the construction of $x(t)$ and it holds without any assumption on the family $\{ \xi_{ij} (s); 1 \leq i,j \leq N , \, s \geq 1 \}$. 

\begin{lem} \label{lemma_coupling_BRW} For $N \geq MK$, let $x(t)$ be the branching/selection process constructed as above. Then, $x(t)$ bounds $X(t)$ from below.  
\end{lem}
\begin{proof} It is immediate that $x(0) \prec X (0)$, hence assume that $x(t) \prec X (t)$. Before the selection step, there are $M K$ points at the positions 
$$
x_i(t) + \xi_{\sigma_i(t), j_l^{(i)} } (t+1) \qquad 1 \leq l \leq K \text{ and } 1 \leq i \leq M.
$$  
By the construction of $x(\cdot)$ the $j_l^{(i)}$ are all distinct. Since $x_i(t) \leq X_{\sigma_i(t)} (t)$, we have that
$$
x_i(t) + \xi_{\sigma_i(t), j_l^{(i)} } (t+1)  \leq X_{\sigma_i(t)} (t) + \xi_{\sigma_i(t), j_l^{(i)} } (t+1) \leq X_{j_l^{(i)}} (t+1).
$$
Hence, the point process obtained from the branching step bounds $X(t+1)$ from below, as a consequence, after the selection step $x(t+1) \prec X(t+1)$, proving the statement.
\end{proof}

Now, we prove that the point processes $\mathcal{L}^{(K)} (\cdot)$ are i.i.d. Lemma~\ref{lemma independence point processes} holds under the unique assumption that the family $\{ \xi_{ij} (s); 1 \leq i,j \leq N , \, s \geq 1 \}$ is i.i.d. 

\begin{lem}\label{lemma independence point processes} Assume that $N \geq KM$, and let $\big\{ \mathcal{T}_i (t); t \in \N ; i=1, \ldots, M \big\}$ be the set of indices obtained in the above construction. For $t \geq 0$ denote by $\Xi\big(x_i(t) \big)$ the point process 
$$
\Xi \big(x_i(t) \big) := \sum_{j \in \mathcal{T}_i(t+1)}  \delta_{\big\{\xi_{\sigma_i (t) ,j} (t+1) \big\}},
$$
then, $\big\{\Xi \big(x_i(t) \big) ; 1  \leq \!  i \! \leq  M; \, t \!\in \! \N \big\}$ are i.i.d. 

In particular, the family of point processes $\left\{ \mathcal{L}^{(K)}\big (x_i(t) \big); \, t \in \N \text{ and } i=1, \ldots, M  \right\}$ is also i.i.d.
\end{lem}

\begin{proof} Note that the families of random variables 
$$
\{ \sigma(s); \, 0 \leq s \leq  t \}, \quad \big\{ \mathcal{T}_i (s) ; 0 \leq s \leq  t; 1 \leq i \leq  M \big\}, \quad \big\{ \Xi \big(x_i(s)\big) ; 1 \leq s \leq  t-1; 1 \leq i \leq  M \big\},
$$
are $\F_t$-measurable with $\F_t$ from (\ref{eq:filtration}). By assumption, $\sigma \big\{ \xi_{ij}(t+1); 1\leq i,j \leq N \big\}$ is independent from $\F_t$, then, by successive conditioning, one easily checks that conditionally on $\F_t$ the vector
$$
\big( \xi_{\sigma_i (t) ,j} (t+1); \, i=1, \ldots, M   \text{ and } j \in \mathcal{T}_i \big)  
$$ 
is distributed according to a $M \times (N-KM)$ vector, whose entries are i.i.d. copies of $\xi_{ij}$, which implies the independence from $\F_t$. Moreover, the conditional independence of the $\xi_{\sigma_i ,j} (t+1)$ yields that $\big(\Xi\big(x_i(t) \big); i=1, \ldots, M  \big)$ are also independent, which proves the first claim.

The second claim is an immediate consequence of the first part of the lemma.
\end{proof}

Finally, we focus on the asymptotic distribution of $\mathcal{L}^{(K)} (\cdot)$ after suitable rescaling,
\begin{equation}\label{equa_definition_coupling_point_process}
\PP^{(N,K)} \big( x_i(t) \big):= \sum_{z \in \mathcal{L}^{(K)} (x_i(t)) } \delta_{\left\{z a_N^{-1}  \right\} },
\end{equation}
for $a_N$ given by  (\ref{equation_definition_aN}). With some abuse of notation, we denote by $\PP^{(N,K)}$ the common distribution of these point processes.

\begin{prop}\label{proposition moment convergence} Assume that the hypothesis of Theorem~\ref{thm:main_theorem} hold with $x_\xi = 0$ and that $M$ and $K$ are fixed. Then, as $N \to \infty$, 
$$\PP^{(N,K)} \longrightarrow \PP^{(K)} \qquad {\rm in\ law,}$$
with $ \PP^{(K)}$ defined in Corollary (\ref{cor:PPK}) with $\beta=\alpha-1$ and $C=\alpha$.
Moreover, for every $\ell > 0$ the moment convergence 
$$
\lim_{N \to \infty} \E \left[ \left|\min \PP^{(N, K)} \right|^\ell \right] = \E \left[ \left|\min \PP^{(K)} 
\right|^\ell \right] < \infty
$$ 
also holds.
\end{prop}
\begin{proof} It suffices to prove the convergence for $ \PP^{(N,K)} \big( x_1(t) \big)$. Since $\xi_{ij}$ is in the domain of attraction of $\Psi_\alpha$ and $x_\xi = 0$, for every $z > 0$, 
\begin{equation} \label{eq:ousertattraction}
\P \left( \xi_{ \sigma_1(t) j} (t) > -z a_N \right)  \sim z^{\alpha} \P \left( \xi_{ \sigma_1(t) j} (t) > - a_N \right)  \sim \frac{z^{\alpha}}{N}  \qquad \quad \text{ as } N \to \infty.
\end{equation}
It is a classical result of extreme value theory  \cite{Resnick1987}  that, as $N \to \infty$, the point process 
$$
\PP^{(N,K)} \stackrel{\rm law}{=}\sum_{j =1}^{N-KM} \delta_{ \left\{ a_N^{-1} \xi_{1, j } (t) \right\} }.
$$
converges in distribution to a PPP with intensity measure $|z|^{\alpha-1} \alpha \textbf{1}_{\{z < 0\}}  \text{d}z$, as claimed.
A necessary and sufficient condition for the convergence of the $\ell$th moment is that the r.v. $\xi_{ij}$ has itself finite $\ell$th moment, which is a consequence of \ref{first_hypotesis}. Proposition~2.1 in \cite{Resnick1987} proves this statement for the maxima of i.i.d. random variables in the domain of attraction of $\Psi_\alpha$. 

Now, a line-by-line adaptation of Proposition~2.1 in \cite{Resnick1987} yields the last claim. 
\end{proof}

A straightforward consequence of Proposition~\ref{proposition moment convergence} and the two previous lemmas is that
\begin{equation} \notag
 \mathpzc{x}^{(N)} (t): = \sum_{i=1}^{M} \delta_{ \left\{ a_N^{-1} x_i (t) \right\} }
\end{equation}
converges in distribution to the point process obtained from a $M$-BRW$\big(\PP^{(K)}\big)$ at time $t$, moreover,
$$
\frac{a_N}{t}\max \left(  \mathpzc{x}^{(N)} (t) \right) \leq \frac{1}{t}\max_{1 \leq j \leq N} \big(  X_j(t) \big).
$$ 
We will prove in Subsection~\ref{subsection velocity BRW} that if one chooses $K$ and $M$ large enough (depending only on $\varepsilon$ and the distribution $\xi_{ij}$), then for $N$ larger than some $N_0 > 0$, 
\begin{equation}\label{equaiton_lowerbound_minimum}
\liminf_{t \to \infty}\frac{1}{t} \max \left(  \mathpzc{x}^{(N)} (t) \right) \geq  -c_\alpha - \varepsilon \qquad a.s.
\end{equation}
which proves the lower bound (\ref{equation lowerbound epsilon}). 

\subsection{Uniform lower bound for the velocities} \label{subsection velocity BRW}

In this subsection, we prove the lower bound (\ref{equaiton_lowerbound_minimum}), which concludes the proof of Theorem~\ref{thm:main_theorem}. The proof is divided in two main steps. 

In the first one, we focus on the BRWs defined by $\PP^{(N,K)}$. We prove that if $N$ is sufficiently large, with positive probability there exists more than $M$ vertices $w$ in generation $n$ (see Subsection~\ref{subsubsection_first_step} for its definition), such that
$$
{\rm position}(w_t) \geq -(c_\alpha +\varepsilon/2) t \qquad \forall w_t \in [\![ e, w ]\!].
$$
In the second step, we use this result to obtain the uniform lower bound (\ref{equaiton_lowerbound_minimum}) for the $M$-BRWs.

Most of the arguments presented here have already been used by B\'erard and Gou\'er\'e \cite{Berard2010}. In our case, though, we deal with a family of BRWs indexed by $N$, whereas in \cite{Berard2010} they compute the velocity for a unique $M$-BRW. 

\subsubsection{First step}\label{subsubsection_first_step}

Let $\PP^{(N,K)}$ be the distribution defined by (\ref{equa_definition_coupling_point_process}) and $\PP^{(K)}$ denote the distribution of a point process obtained from the $K$ largest points of a PPP with intensity measure $ |z|^{\alpha-1} \alpha \mathbf{1}_{\{z \leq 0\}} \text{d}z$. Then, BRW$\big(\PP^{(N,K)}\big)$ and BRW$\big(\PP^{(K)}\big)$ generate the same Galton-Watson tree, in which every individual has a constant number $K$ of offspring, denote by $\mathbb{T}_K$ this tree. We will construct these BRWs on a same probability space. 

Let $\{ \PP^{(N,K)} (w); w \in \mathbb{T}_K \}$ be i.i.d. copies of $\PP^{(N,K)}$, $\{ \PP^{(K)} (w); w \in \mathbb{T}_K \}$ be i.i.d. copies of $\PP^{(K)}$, and $\big(\Omega, \F, \P \big)$ be a probability space where those families of r.v. are defined. Since $\PP^{(N,K)}$ converges to $\PP^{(K)}$ in distribution (see Proposition~\ref{proposition moment convergence}), we can and we will assume that the stronger a.s. convergences
\begin{equation}
\label{eq:assextra}
\lim_{N \to \infty} \PP^{(N,K)} (w) = \PP^{(K)}(w)  \quad a.s.
\end{equation}
hold for all $w \in \mathbb{T}_K$, which implies the point-to-point convergence 
$$
\lim_{N \to \infty} \left\|\PP^{(N,K)} \big( w \big) - \PP^{(K)}\big( w \big) \right\| = 0 \qquad a.s.  
$$
where $\|\cdot \|$ is the distance defined in (\ref{measure distance}). Note that we have not lost in generality, since we can always construct a probability space $\big(\Omega, \F, \P \big)$, for which the a.s. convergence holds.

Under these hypothesis, the construction goes as follows. Each individual $w \in \mathbb{T}_K$ has $K$ offspring, that we label according to some deterministic order. Let $w^{(i)}$ be its $i$th children, then, its position $ \eta^{(N)}(ww^{(i)})$ and $\eta^{(\infty)}(ww^{(i)})$ in the BRW$\big(\PP^{(N,K)}\big)$ and BRW$\big(\PP^{(K)}\big)$ are given by
$$
\eta^{(N)}\big(ww^{(i)} \big) = \eta^{(N)}(w) + \PP^{(N,K)}_i(w) \quad  \text{ and } \quad \eta^{(\infty)}\big(ww^{(i)} \big)= \eta^{(\infty)}(w) + \PP^{(K)}_i(w),
$$ 
where  $\PP^{(N,K)}_i(w)$ and $ \PP^{(K)}_i(w)$ denote the $i$th largest point in  $\PP^{(N,K)}(w)$ and $ \PP^{(K)}(w)$ respectively. This construction couples the BRWs and for $w \in \mathbb{T}_K$ fixed
$$
\lim_{N \to \infty} \eta^{(N)}(w)= \eta^{(\infty)}(w) \qquad a.s.  
$$

A direct calculation shows that $\PP^{(K)}$ satisfies (\ref{equation reproductive assumption}--\ref{equation log generating function relation}), which implies the existence of the asymptotic velocity $\gamma\big(\PP^{(K)} \big)$, with $\gamma$ given by
(\ref{equation convergence classical BRW}). Lemma~\ref{Poisson_lemma} with $C=\alpha$ and $\beta = \alpha -1 > -1$ yields 
$$
\lim_{K \to \infty} \gamma \big( \PP^{(K)} \big) = - \frac{\alpha}{\e} \left( \frac{1}{\Gamma(\alpha) \alpha } \right)^{\frac{1}{\alpha}} = - c_\alpha .
$$ 
Let $\delta = \varepsilon/12$, then there exists $K_0$ such that $\forall K \geq K_0$ 
\begin{equation}\label{equa_defi_K}
\gamma \big(\PP^{(K)} \big) \geq -c_\alpha - \delta.
\end{equation}
Fix $K$ for which (\ref{equa_defi_K}) holds; {\it we bring to the reader's attention that, as  \ref{first_hypotesis}, \ref{second_hypotesis} hold, $\PP^{(N,K)}$ also satisfies (\ref{equation reproductive assumption}), (\ref{equation log generating function assumption}). Moreover, a simple calculation shows that $\gamma\big(\PP^{(N,K)}\big)$ converges to $\gamma\big(\PP^{(K)}\big)$ as $N \to \infty$.}
\medskip

We now prove that with positive probability there exists more than $M$ individuals $\tilde{w} \in \mathbb{T}$ in generation $n$ such that
$$
\eta^{(N)} \big( \tilde{w}_t \big) \geq -c_\alpha t - 6\delta, \qquad \text{for every } \tilde{w}_t \in [\![ e, \tilde{w} ]\!].
$$
As it will become clearer in the sequel, we take $n$ of the form $n=s_M + m$, with
\begin{equation}\label{equation_defintion_SM_m}
s_M : = \left\lceil  \frac{\ln M}{\ln \varphi} \right\rceil +1 \quad \text{and} \quad m = \left\lceil \frac{\big(|R|-c_\alpha - 6 \delta \big) s_M }{3 \delta }  \right\rceil.
\end{equation}
The constants $\varphi > 1$ and $R < -c_\alpha - 6 \delta < 0 $ are given by Lemma~\ref{lemma galton watson tree} and formula (\ref{equation defintion R}) below and depend only on the distribution $\PP^{(K)}$. Although $M$ may be very large, it will be kept constant throughout this section (while $N\to \infty$), hence $s_M$ and $m$ are also constants. 

First, we obtain a lower bound for the probability of the set 
$$
\Big\{  \exists{w} \in \mathbb{T}_K \text{ in generation $m$ such that } \eta^{(N)} \big( w_t \big) \geq (-c_\alpha - 3 \delta) t; \, \forall w_t \in [\![ e, w ]\!] \Big\}.
$$
Denote by $A_{m, \delta}$ the set
$$
A_{m, \delta} := \Big\{ \left\|\PP^{(N,K)} \big(w' \big) - \PP^{(K)} \big(w' \big) \right\| \leq \delta ; \forall w' \in \mathbb{T}_K \text{ such that } |w'| \leq m   \Big\}
$$
then, for $m \in \N$ and $\delta$ fixed, one obtains from (\ref{eq:assextra}) that $\P(A_{m, \delta}) \to 1$ as $N \to \infty$. Since $\gamma\left(\PP^{(K)} \right) \geq c_\alpha  - \delta $ we have the following set inclusions
\begin{align*}
\big\{ &  \exists{ w }\in \mathbb{T}_K \text{ such that $|w|=m$ and }  \eta^{(N)} \big( w_t \big) \geq (-c_\alpha - 3 \delta) t; \, \forall w_t \in [\![ e, w ]\!] \big\} \\
& \supset \big\{ \exists{ w} \in \mathbb{T}_K \text{ such that $|w|=m$ and }  \eta^{(N)} \big( w_t \big) \geq (-c_\alpha - 3 \delta) t; \, \forall w_t \in [\![ e, w ]\!] \big\} \cap A_{m, \delta}  \\
&\supset \big\{\exists{w} \in \mathbb{T}_K \text{ such that $|w|=m$ and } \eta^{(\infty)} \big( w_t \big) \geq (-c_\alpha - 2 \delta) t; \, \forall w_t \in [\![ e, w ]\!] \big\} \cap A_{m, \delta} \\
&\supset \big\{\exists{w} \in \mathbb{T}_K \text{ such that $|w|=m$ and } \eta^{(\infty)} \big( w_t \big) \geq (-\gamma \big(\PP^{(K)} \big) - \delta) t; \, \forall w_t \in [\![ e, w ]\!] \big\} \cap A_{m, \delta} \\
& \supset \big\{\text{BRW$\big(\PP^{(K)}\big)$ has an infinite ray lying above the line of slope } (\gamma(\PP^{(K)})\!-\! \delta)  \big\} \cap A_{m, \delta},
\end{align*}  
which yields the lower bound
$$
\liminf_{N \to \infty} \P\Big(  \exists{w} \in \mathbb{T}_K \text{ such that $|w|=m$ and }  \eta^{(N)} \big( w_t \big) \geq (-c_\alpha - 3 \delta) t; \, \forall w_t \in [\![ e, w ]\!] \Big) \geq \rho(\infty, \delta).
$$
From Theorem~\ref{theorem gantert hu shi}, $\rho(\infty, \delta) > 0$ is a constant depending only on $\PP^{(K)}$. Then, there exists $N_M \in \N$ depending only on $m$ (and hence, on $M$) such that $\forall N \geq N_M$
\begin{equation} \notag
\P\Big(  \exists{w} \in \mathbb{T}_K : |w|=m,  \eta^{(N)} \big( w_t \big) \geq (-c_\alpha - 3 \delta) t, \, \forall w_t \in [\![ e, w ]\!] \Big) \geq \frac{\rho(\infty, \delta)}{2}.
\end{equation}

Now, we choose $R$ and $\varphi$ in (\ref{equation_defintion_SM_m}). Since $\PP^{(K)} \big( ]-\infty, 0[ \big)= K$, one can take $R<-c_\alpha -6\delta< 0$ such that 
\begin{equation} \label{equation defintion R}
\P \left( \PP^{(K)} [R, 0) \geq 2 \right) > \frac{2}{3}.
\end{equation}
Using the convergence in distribution, there exists a $N' > 0$ such that for $N \geq N'$
$$
\P \left( \PP^{(N,K)} [R, 0) \geq 2 \right) \geq \frac{2}{3}.
$$ 
Without loss of generality, we can and we will assume that $N_M \geq N'$. The Galton-Watson tree whose offspring distribution is 
\begin{equation}\label{galton watson offspring distribution}
p_i^{(N)} = \P \left( \PP^{(N,K)} [R, 0) = i \right), \qquad i=0,1,\ldots,
\end{equation}
has mean offspring larger than $4/3$. It is supercritical, and the following well-known result holds. 

\begin{lem}[\cite{Athreya2004} Theorem 2 Section~6 Chapter~1]
 \label{lemma galton watson tree} Let $M_t$ denote the population size of a supercritical Galton-Watson process with square integrable offspring distribution (started with one individual). Then, there exists $r> 0$ and $\varphi >1$ such that for all $t \geq 0$
$$
\P(M_t \geq \varphi^t) \geq r.
$$
\end{lem}

Let $M_t^{(N)}$ denote the population size of the Galton-Watson processes defined by $(p_i^{(N)})_{i= 0, 1,\ldots}$. Using a simple coupling argument and Lemma~\ref{lemma galton watson tree}, we can find a $\varphi >1$ and $r>0$ not depending on $N \geq N_M$ such that for all $t \geq 1$
$$
\P(M_t^{(N)} \geq \varphi^t ) > r.
$$
With $m$ and $s_M $ from (\ref{equation_defintion_SM_m}), we have that
$$M_{s_M}^{(N)} \geq M; \qquad {\rm with\ probability\ at\ least\ }r > 0,$$
and that
$$
(-c_\alpha - 3 \delta) m + R(t-m) \geq (-c_\alpha -6  \delta) t; \qquad \text{ for every } m \leq t \leq m+s_M .
$$
Let $ww' \in \mathbb{T}_K$ be a vertex in generation $n$, with the following properties: $|w| = m$,
$$
\eta^{(N)} \big(w_{t} \big) > (-c_\alpha - 3 \delta) t, \qquad \forall w_t \in [\![ e, w ]\!],
$$ 
$w'$ is in generation $s_M$ in the $\mathbb{T}_K$ sub-tree descending from $w$ and
$$
\eta^{(N)} \big(w'_{s+1} \big) - \eta^{(N)} \big(w'_{s} \big) \geq R, \qquad \forall w'_s \in [\![ w, w' ]\!] . 
$$ 
Then, by a simple calculation one can conclude that the path $[\![ e, ww' ]\!] \subset \mathbb{T}_K$ has always lain above the line of slope $-c_\alpha -6  \delta$. For $N \geq N_M$, a conditioning argument yields the lower bound for the probabilities 
\begin{align*}
\P & \left(  \sharp \left\{ \tilde{w} \in \mathbb{T}_K; \, |\tilde{w}|=n \text{ and } \eta^{(N)} (\tilde{w}_t) \geq -\left(c_\alpha + \frac{\varepsilon}{2}\right) t \; \forall \tilde{w}_t \in [\![e,\tilde{w}]\!] \right\} \geq M \right) \\
&  \geq \P  \Big( \exists{w} \in \mathbb{T}_K \text{ such that } \eta^{(N)} (w_t) \geq (-c_\alpha - 3 \delta) t; \, \forall w_t \in  [\![ e, w ]\!] \text{ and } M^{(N)}_{s_M} \geq M \Big) \\
& \geq r \frac{\rho(\infty, \delta)}{2},
\end{align*}
in the second equation, $w$ is a vertex in generation $m$ and $M^{(N)}_t$ is the population size of the Galton-Watson process generated by the descendants of $w$ for which  
$$
\eta^{(N)} \big(w'_{s+1} \big) - \eta^{(N)} \big(w'_{s} \big) \geq R \qquad \forall w'_s \in  [\![ w, w' ]\!] . 
$$ 
In particular, we have just proved the following proposition.

\begin{prop}\label{proposition BRW M particles above the slope} Let $\big(\eta^{(N)} (w); \, w \! \in \! \mathbb{T}_{K} \big) $ the BRW defined by the point processes $\PP^{(N,K)}$. Given $\varepsilon > 0$ let $R$ be given by (\ref{equation defintion R}), $r$ and $\varphi$ as in Lemma~\ref{lemma galton watson tree}. Then, take $s_M$ and $m$ as in (\ref{equation_defintion_SM_m}). Then, with $n= m+s_M$,  we can find some $N_M$ (depending only on $M$) such that if $N \geq N_M$
$$
\P\left( \sharp \left\{ w \in \mathbb{T}_K; \, |w|=n \text{ and } \eta^{(N)} (w_t) \geq -\left(c_\alpha + \frac{\varepsilon}{2}\right) t ; \, \forall w_t \in [\![e,w]\!] \right\} \geq M  \right) \geq \frac{\rho(\8, \varepsilon/12)}{2} r
$$
\end{prop} 

\subsubsection{Second step: uniform lower bound for the speed}\label{speed M BRW}

In this step, we obtain a uniform lower bound for the speed of $a_N^{-1} x(t)$, which is simply a M-BRW$\big(\PP^{(N,K)}\big)$. Let $\mathpzc{x}^{(N)} (t)$ be the point process associated to $a_N^{-1} x(t)$
$$
\mathpzc{x}^{(N)} (t) := \sum_{i=1}^{M} \delta_{\{a_N^{-1}x_i(t) \}},
$$ 
and $\gamma_M \big(\PP^{(N,K)} \big)$ be the asymptotic velocity of the $M$-BRW$\big(\PP^{(N,K)} \big)$,
$$
\lim_{t \to \infty }t^{-1} \min \left( \mathpzc{x}^{(N)} (t) \right)=\lim_{t \to \infty }t^{-1} \max \left( \mathpzc{x}^{(N)} (t) \right) = \gamma_M(\PP^{(N,K)}) \qquad a.s.
$$
Then, we will prove that for $\varepsilon > 0$ and $K$ given by (\ref{equa_defi_K}) the inequality
$$
 \liminf_{M\to \infty} \left( \liminf_{N\to \infty} \gamma_M(\PP^{(N,K)}) \right) \geq -(c_\alpha + \varepsilon)
$$ 
holds, which finishes the proof of Theorem~\ref{thm:main_theorem}. Following the strategy of \cite{Berard2010}, we construct a third point process $W(t)= W^{(N)}(t)$ that bound $\mathpzc{x}^{(N)} (t) $ from below. This new point process evolves like $\mathpzc{x}^{(N)} (t) $ up to a certain random time $\tau_i, \, i \in \N$, from which we shift the position of all particles to the minimal position, and start $W(t)$ afresh.  

Let $n = m+s_M$, where $s_M$ and $m$ are given by  (\ref{equation_defintion_SM_m}). We will construct the process $W(t)$ and the stopping times $0=\tau_0 < \tau_1 < \ldots $ together
\begin{equation} \notag
\tau_1:= \inf \left\{ 1 \leq s \leq n; \, \min \left( \mathpzc{x}^{(N)} (s) \right) \geq (-c_\alpha - \varepsilon/2) s  \right\},
\end{equation} 
where $\inf\{ \emptyset \} = n $. Then, $\tau_1 \leq n$ is a stopping time with respect to the filtration $\F_t$. For $0 \leq t \leq \tau_1$ let
$$
W(t) = \mathpzc{x}^{(N)} (t). 
$$
and $ m_1 := \min \big(W (\tau_1) \big)$, then at the time step $\tau_1 \to \tau_1 +1$ we shift all particles $W_i$ to $m_1$ and continue the construction up to $\tau_2$ according to the induction step.

\emph{Inductive step:} assume that $\tau_1 \! < \! \ldots \! < \! \tau_l$ and $W (t)$ for $t \! \leq \! \tau_l$ are defined. Then, for $\tau_l \!+\! 1\! \leq \! t \! \leq \! \tau_{l+1}$ (we will define $ \tau_{l+1}$ below), $W(t)$ is the point process of a $M$-BRW$\big(\PP^{(N,K)}\big)$ starting from
$$
m_{l} := \min\big(W (\tau_l) \big).
$$ 
At each time step $t \to t+1$ the individuals $\big(W_i(t)\big)_{i=1,\ldots,M}$ give birth to $K$ new individuals, whose positions are determined by independent point process $\big(\PP^{(N,K)} (x_i(t)); i=1,\ldots,M  \big)$, and die immediately afterwards. We assume that the point process defining $\mathpzc{x}^{(N)}$ and $W(\cdot)$ are the same. Moreover, we will also assume that the indices are organized in order to couple $\mathpzc{x}^{(N)}$ by $W(\cdot)$. We then select the $M$ rightmost particles to form the next generation. 

The process evolves as above up to
\begin{equation} \notag
\tau_{l+1}:= \inf \big\{ \tau_{l}+1 \leq s \leq \tau_{l} + n; \, \min \left( W(s) \right) -m_l \geq (-c_\alpha - \varepsilon/2) s  \big\},
\end{equation} 
where we shift the positions of the $M$ particles to $m_{l+1}$, the minimum of the positions. It is immediate from the construction of $W(\cdot)$ that   
$$
W(t) \prec \mathpzc{x}^{(N)} (t) \qquad \forall t \in \N.
$$
For $l \geq 1$, the processes $\big( W(t)-m_l; \  t \in [\tau_l+1, \tau_{l+1}]  \big)$ and the random variables $\tau_{l+1} - \tau_l$ are i.i.d. In the sequel, we use the notation $\tau := \tau_1$, then by the law of large numbers
$$
\lim_{l\to \infty} \frac{1}{l} \min \left( \mathpzc{x}^{(N)} (\tau_l) \right) = \gamma_M \big(\PP^{(N,K)}\big) \E[\tau] \qquad a.s.
$$ 
From the construction of $W(\cdot)$ and the renewal theorem we also obtain that
$$
\liminf_{l\to \infty} \frac{1}{l} \min \left( \mathpzc{x}^{(N)} (\tau_l) \right) \geq 
\liminf_{l\to \infty} \frac{1}{l} \min \big( W(\tau_l) \big) =
\E\left[ \min \left( W(\tau) \right) \right] \quad a.s.
$$ 
%
which implies that
\begin{equation}
\label{eq:1fin}
\gamma_M \big(\PP^{(N,K)} \big) \geq  \frac{\E\left[ \min \big( W(\tau) \big) \right]}{\E[\tau]}.
\end{equation}
With $B= \big\{ \min \big( W(\tau)\big) < (-c_\alpha - \varepsilon/2) \tau  \big\}$,
we write
\begin{align*}
\min \big( W(\tau) \big) & \geq (-c_\alpha - \varepsilon/2) \tau \textbf{1}_{B^\complement} + \min \big(W(n)\big) \textbf{1}_{B}  \\
& = (-c_\alpha - \varepsilon/2) \tau + (c_\alpha + \varepsilon/2) \tau \textbf{1}_{B} \notag + \min \big( W(n) \big)  \textbf{1}_{B}.
\end{align*}
Taking expected value we get
\begin{equation}
\label{eq:2fin}
\E\big[\min \big( W(\tau) \big)\big]  \geq (-c_\alpha - \varepsilon/2) \E[\tau] + \E \big[ \min \big( W(n) \big) \textbf{1}_{B} \big]. 
\end{equation}
Let $\min \Big( \PP^{(N,K)} \big(W_i(t) \big) \Big)$ be the smallest point of the point process generated by $W_i(t)$ before the selection step, it has the law of the $K$th maxima of a $N-KM$ sample of $\xi_{ij}$. Since $\xi_{ij} \leq 0$, one gets the lower bound 
$$
\min \big( W(n) \big) \geq   \sum_{t=0}^{n} \sum_{i=1}^{M} \min \Big( \PP^{(N,K)} \big(W_i(t) \big) \Big),
$$
which implies that
$$
\E\big[\min \big( W(n) \big) \textbf{1}_{B}\big] \geq -(n+1) M \E \left[ \left| \min \left( \PP^{(N,K)} \right) \textbf{1}_{B} \right|  \right].
$$ 
The Cauchy-Schwarz inequality yields
$$
\E[\min W(n) \textbf{1}_{B}] \geq -(n+1) M \E \left[ \left|\min \left( \PP^{(N,K)} \right) \right|^2  \right]^{1/2} \P(B)^{1/2}.
$$
By Proposition~\ref{proposition moment convergence}, the second moment of $\min \left( \PP^{(N,K)} \right) $ converges as $N \to \infty$ to a finite constant. Hence, there exists a constant $\tilde{c}$, depending only on $\xi_{ij}$, such that 
$$
\E \big[\min W(n)\textbf{1}_{B} \big] \geq -\tilde{c} (n+1) M   \P(B)^{1/2}.
$$
Finally, the probability of $B$ can be estimated using Proposition~\ref{proposition BRW M particles above the slope}. The evolution of different individuals in the $M$-BRW is not independent. Yet, a $M$-BRW can be coupled with $M$ independent BRWs, see Section~3.3 in \cite{Berard2010}, so that the event ``the minimum of the $M$-BRW always lies below the line of slope $-c_\alpha - \varepsilon/2$" implies that none of the $M$ independent BRWs has more than $M$ vertices in generation $n$ that have always stayed above this line, hence
$$
\P(B) \leq \left( 1- \frac{\rho(\infty, \varepsilon/12)}{2} r \right)^{M}.
$$
From the definition of $n$, if $M$ is large enough $\tilde{c} (n+1) M < M^2$ and
$$
\limsup_{M \to \infty} M^2 \left( 1- \frac{\rho(\8,\varepsilon/12)}{2} r \right)^{M/2}  =0.
$$
Then, choosing $M$ properly (note that it depends only on $\varepsilon$ and $\xi_{ij}$ but not on $N$), one gets
$$
\E [ \min W(n) \textbf{1}_{B}] \geq -\frac{\varepsilon}{2}.
$$
Combined with (\ref{eq:1fin}, \ref{eq:2fin}), this ends the proof.

\bigskip

{\bf Acknowledgements:} The authors thank Bernard Derrida for stimulating conversations 
on this model.

\bigskip

{\footnotesize
\bibliographystyle{plain}
\bibliography{density_velocity}
}

\end{document}